\documentclass[a4paper,11pt]{amsart}
\usepackage{amssymb}
\usepackage[alwaysadjust]{paralist}
\usepackage[capitalize]{cleveref}
\numberwithin{equation}{section}

\newtheorem{cor}[equation]{Corollary}
\newtheorem{lem}[equation]{Lemma}
\newtheorem{prop}[equation]{Proposition}
\newtheorem{thm}[equation]{Theorem}
\theoremstyle{definition}

\newtheorem{exas}[equation]{Examples}
\newtheorem{rem}[equation]{Remark}
\newtheorem{rems}[equation]{Remarks}

\def\R{\mathbb R}

\def\Z{\mathbb Z}
\def\ve{\varepsilon}
\def\vf{\varphi}

\newcommand{\id}{\operatorname{id}}

\newcommand{\spec}{\operatorname{spec}}
\newcommand{\supp}{\operatorname{supp}}

\begin{document}

\title{Small eigenvalues of surfaces of finite type}
\author{Werner Ballmann}
\address
{Max Planck Institute for Mathematics,
Vivatsgasse 7, D--53111 Bonn\\}
\email{ballmann\@@mpim-bonn.mpg.de}
\author{Henrik Matthiesen}
\address
{Max Planck Institute for Mathematics,
Vivatsgasse 7, D--53111 Bonn\\}
\email{hematt\@@mpim-bonn.mpg.de}
\author{Sugata Mondal}
\address{Max Planck Institute for Mathematics,
Vivatsgasse 7, D--53111 Bonn\\}
\email{sugata.mondal@mpim-bonn.mpg.de}

\thanks{\emph{Acknowledgments.}
We thank Eberhard Freitag, Ursula Hamenst\"adt, and Werner M\"uller
for helpful comments
and gratefully acknowledge the support and the hospitality of the Max Planck
Institute for Mathematics in Bonn, the Hausdorff Center for Mathematics in Bonn,
and the Erwin Schr\"o\-dinger Institute in Vienna.}

\date{\today}


\begin{abstract}
Extending our previous work on eigenvalues of closed surfaces
and work of Otal and Rosas,
we show that a complete Riemannian surface $S$ of finite type
and Euler characteristic $\chi(S)<0$
has at most $-\chi(S)$ small eigenvalues.
\end{abstract}

\maketitle

\section{Introduction}
\label{intro}

For Riemannian metrics on the closed surface $S=S_g$ of genus $g\ge2$,
the eigenvalue $\lambda_{2g-2}=\lambda_{-\chi(S)}$ plays a specific role.
On the one hand, Buser gave examples of hyperbolic metrics
on $S$ such that the first $2g-2$ eigenvalues
\begin{equation*}
  0 = \lambda_0 < \lambda_1 \le \dots \le \lambda_{2g-3}
\end{equation*}
are arbitrarily small \cite[Satz 1]{Bu1}.
On the other hand, Schoen, Wolpert, and Yau proved that there is a
constant $c=c(g)>0$ such that $\lambda_{2g-2} > c$
for any Riemannian metric on $S$ with curvature $K\le-1$ \cite{SWY}.
Buser then showed that, for hyperbolic metrics, the constant $c$
can be chosen to be independent of the genus \cite[Theorem 8.1.4]{Bu2}.
This development culminated in the work of Otal and Rosas,
who showed that $\lambda_{2g-2}>\lambda_0(\tilde S)$
for any analytic Riemannian metric on $S$ with curvature $K\le-1$,
where $\lambda_0(\tilde S)$ denotes the bottom of the spectrum
of the universal covering surface of $S$,
endowed with the lifted Riemannian metric \cite[Th\'eor\`eme 1]{OR}.

Recall that the bottom of the spectrum of the hyperbolic plane is $1/4$
and that we have $\lambda_0(\tilde S)\ge1/4$ if $K\le-1$;
see also \eqref{bottom} below.

Dodziuk, Pignataro, Randol, and Sullivan extended the work of Schoen, Yau, and Wolpert
to the non-compact surfaces $S_{g,p}$ of genus $g$ with $p>0$ punctures
(where $2g+p>2$).
They showed that there is a constant $c=c(2g+p)$ such that complete hyperbolic
metrics on $S_{g,p}$ --of finite or infinite area--
have at most $2g+p-2$ eigenvalues $\lambda$, counted with multiplicity,
with $\lambda\le c$ \cite[Corollary 1.3]{DPRS}.
In \cite[Th\'eor\`eme 2]{OR},
Otal and Rosas improve this for complete hyperbolic metrics of finite area to $c=1/4$. 

At the end of their article, Otal and Rosas discuss the question
whether their results also hold for smooth Riemannian metrics.
In our previous article \cite{BMM}, we showed this for closed surfaces
and sharpened their lower bound $\lambda_0(\tilde S)$. 
In the present article, we generalize their results
to \emph{surfaces of finite type}, more precisely, to surfaces $S$
with compact boundary (possibly empty) with $-\infty<\chi(S)<0$
and complete Riemannian metrics on them (possibly of infinite area),
again with a sharper lower bound.

Recall that a surface $S$ with compact boundary is of finite type if and only if
it is diffeomorphic to a closed surface with $p\ge0$ points and $q\ge0$ open discs removed.
Then $S$ has $p$ ends, represented by the punctures,
and $q$ boundary circles, the boundaries of the deleted open discs.
Note that we are only concerned with the diffeomorphism type of $S$.
Thus a puncture has the same effect as the removal of a closed disc.

A basis of the neighborhoods of an end of $S$
consists of punctured discs around the corresponding deleted point.
We call these punctured discs \emph{funnels} and visualize the surface as a steamboat
with the funnels pointing upwards and the rest of the surface below them.
As already emphasized above, we do not distinguish between different conformal types.
For example, in our terminology, a hyperbolic cusp is a funnel.

We assume that $S$ is endowed with a Riemannian metric
which is complete with respect to the associated distance function.
The area of the metric may be finite or infinite.
We view the Laplacian $\Delta$ of $S$ as an unbounded operator
on the space $L^2(S)$ of square integrable functions on $S$
with domain the space of smooth functions on $S$ with compact support
in the interior $\mathring S$ of $S$.
Our concern is the spectrum of the Friedrichs extension of $\Delta$,
which we call the \emph{spectrum} of $S$.
If the boundary of $S$ is empty, a case which we include in our discussion,
this is the usual spectrum of $S$.
Otherwise it is the Dirichlet spectrum of $S$.

For any Riemannian manifold $M$, with or without boundary, denote by $\lambda_0(M)$
\emph{the bottom of the spectrum} of the Laplacian on $M$; that is,
\begin{equation}\label{bottom}
  \lambda_0(M) = \inf R(\vf),
\end{equation}
where $\vf$ runs over all non-zero smooth functions on $M$ with compact support
in the interior of $M$ and $R(\vf)$ denotes the \emph{Rayleigh quotient} of $\vf$, 
\begin{equation}\label{raydef}
  R(\vf) = \frac{\int_M |\nabla\vf|^2}{\int_M\vf^2}.
\end{equation}
As we mentioned above, the bottom of the spectrum of the hyperbolic plane is $1/4$.
The bottom of the spectrum of the Euclidean plane is $0$.

To state the main result of the present article,
we need to introduce one more notion.
Let $S$ be a surface of finite type, with or without boundary,
endowed with a complete Riemannian metric.
Set
\begin{equation}
  \Lambda(S) = \inf_\Omega \lambda_0(\Omega),
\end{equation}
where the infimum is taken over all domains $\Omega$ in $S$
which are diffeomorphic to an open disc, annulus, or cross cap.
Note that any such domain can be lifted to $\tilde S$ or a cyclic quotient of $\tilde S$,
and hence we have
\begin{equation}\label{bro}
  \Lambda(S) \ge \lambda_0(\tilde S)
\end{equation}
by a result of Brooks \cite[Theorem 1]{Br2} (see  Remark \ref{rems}.\ref{broo} below).
In \cite{BMM} we showed that, on a closed surface $S$ with $\chi(S)<0$,
a Riemannian metric has at most $-\chi(S)$ eigenvalues $\lambda$
which are \emph{small} in the sense of $\lambda\le\Lambda(S)$.
The main result of this article is an extension of the latter result
to surfaces of finite type.

\begin{thm}\label{small}
A complete Riemannian metric on a surface $S$ with compact boundary (possibly empty)
and with $-\infty<\chi(S)<0$ has at most $-\chi(S)$ eigenvalues $\lambda$,
counted with multiplicity, with $\lambda\le\Lambda(S)$.
\end{thm}

With $p,q$ as further up,
the case $p=q=0$ corresponds to closed surfaces, treated in \cite{BMM}.
The case $p>0$, $q=0$ (with orientable $S$)
extends \cite[Corollary 1.3]{DPRS} of Dodziuk, Pignataro,
Randol, and Sullivan and \cite[Th\'eor\`eme 2]{OR} of Otal and Rosas
to arbitrary complete Riemannian metrics on such surfaces.
The case $p=0$, $q>0$ corresponds to the Dirichlet spectrum
of compact surfaces with non-empty boundary.
\cref{small} implies also the following extension of the above result
of Dodziuk, Pignataro, Randol, and Sullivan.

\begin{cor}\label{coro}
Let $S$ be a surface of finite type with compact boundary (possibly empty),
endowed with a complete hyperbolic metric of infinite area
such that the boundary of $S$ is weakly convex, that is,
such that the geodesic curvature of the boundary
with respect to the inner normal is non-negative.
Then $S$ has at most $-\chi(S)$ eigenvalues,
counted with multiplicity.
\end{cor}

As for the proof of \cref{coro},
note that the weak convexity of the boundary implies that the shortest curves
in the free homotopy classes of the boundary circles
are closed hyperbolic geodesics in $S$.
Cutting away the pieces between these and the corresponding boundary circles,
we arrive at a hyperbolic surface $S'$ with closed hyperbolic geodesics as boundary.
Then we can decompose $S'$ in the standard way into pairs of pants,
some of them possibly with hyperbolic cusps, and expanding funnels of the kind
$\{(x,y)\mid x\ge0,\,y\in\R/L\Z\}$ with hyperbolic metric $dx^2+\cosh(x)^2dy^2$.
Since the area of $S$ is infinite, at least one expanding funnel occurs.
Hence Theorem 4.8 of \cite{LP} applies and shows that $S$ does not have
eigenvalues $\ge1/4$.
(Note that Theorem 4.8 also applies to surfaces;
see the last sentence in Section 4 of \cite{LP}.)
On the other hand, we have $\Lambda(S)\ge\lambda_0(\tilde S)$
by \eqref{bro} and $\lambda_0(\tilde S)=1/4$.
Now \cref{coro} follows from \cref{small}.

The situation for complete hyperbolic metrics of finite area
is much more complicated;
see e.g.\ Section 2 and Conjecture 1 in \cite{Sa}.

\begin{rems}\label{rems}
\begin{inparaenum}[1)]
\item\label{buse}
The bound $-\chi(S)$ in \cref{small} and \cref{coro} is optimal.
Indeed, the construction of Buser in \cite{Bu1} applies to surfaces $S$
with compact boundary (possibly empty) and $-\infty<\chi(S)<0$
and shows that, for any $\ve>0$, there is a complete hyperbolic metric
on any such $S$ with closed hyperbolic geodesics as boundary circles
such that $S$ has (at least) $-\chi(S)$ eigenvalues $\lambda$,
counted with multiplicity, with $\lambda<\ve$.
Furthermore, if $S$ is not compact,
the metric can be chosen to have finite or infinite area.

\item\label{broo}
Under a Riemannian covering of complete and connected Riemannian manifolds,
the bottom of the spectrum of the covered manifold is at most the bottom of the spectrum
of the covering manifold; see e.\,g.\,\cite[p.\,101]{Br2}.
Brooks showed that, under a normal Riemannian covering of complete and connected
Riemannian manifolds with an amenable group of covering transformations,
the bottom of the spectrum does not change \cite[Theorem 1]{Br2}.
Now the relevant arguments of Brooks in the proof of Theorem 1 in \cite{Br2}
and of Sullivan in the proof of Theorem 2.1 in \cite{Su} remain valid
in the more general case of complete Riemannian manifolds with boundary,
and thus \eqref{bro} follows.

In general, we do not have $\Lambda(S)>\lambda_0(\tilde S)$.
For example, if $S$ is a non-compact complete hyperbolic surface of finite type,
then $\Lambda(S)=\lambda_0(\tilde S)=1/4$.
However, the inequality is strict if $S$ is closed and hyperbolic \cite{Mo}.
More generally, it is strict for any compact Riemannian surface
with negative Euler characteristic, see \cite{BMM2}.

\item\label{ess}
Besides $\lambda_0(\tilde S)$ and $\Lambda(S)$,
there is another constant which is of interest in our context.
Recall that the spectrum of $S$ is the disjoint union
of its discrete and essential parts; see \cref{anapre}.
Denote by $\lambda_{\rm ess}(S)$ the bottom of the essential spectrum of $S$.
Since funnels in surfaces of finite type are diffeomorphic to open annuli,
we have $\Lambda(S)\le\lambda_{\rm ess}(S)$.

For non-compact complete hyperbolic surface of finite type, equality holds.
However, any non-compact surface $S$ of finite type with $\chi(S)<0$
carries complete Riemannian metrics with $\Lambda(S)<\lambda_{\rm ess}(S)$
and an arbitrary large number of eigenvalues $<\lambda_{\rm ess}(S)$;
see Example \ref{exesp}.\ref{many}.
In Example 4.1 of \cite{BCD}, Buser, Colbois, and Dodziuk construct examples
of hyperbolic surfaces $S$ of infinite type which have infinitely many
eigenvalues $<\lambda_{\rm ess}(S)$.

\item\label{ess2}
In Examples \ref{exesp}.\ref{emesp} and \ref{exesp}.\ref{emesp2}
we show that non-compact surfaces of finite type with compact boundary
carry complete Riemannian metrics with $K\le-1$, of finite and infinite area,
with empty essential spectrum.
Such metrics have infinitely many eigenvalues.
\end{inparaenum}
\end{rems}

In the proof of our main result,
our line of arguments is different from the classical one
of Buser \cite{Bu1},  Schoen, Wolpert, Yau \cite{SWY},
and Dodziuk, Pignataro, Randol, Sullivan \cite{DPRS},
who rely on decompositions of the surface into appropriate pieces
and monotonicity properties of eigenvalues.
We follow the strategy of Otal and Rosas in \cite{OR},
which involves a careful examination of topological properties of the nodal lines
and domains of finite linear combinations of eigenfunctions.
In our situation of smooth Riemannian metrics,
such nodal lines and domains may not be as regular as
in the case of analytic Riemannian metrics as considered by Otal and Rosas,
where eigenfunctions are analytic, hence also finite linear combinations of them.
We investigate approximate nodal lines and domains instead and,
for that reason, have to face a number of additional problems
before we get the main argument of Otal and Rosas to work.
This line of proof requires extending our corresponding arguments in \cite{BMM}
from closed surfaces to surfaces of finite type.
Moreover, in the non-compact case,
Otal and Rosas use the rather special behaviour of nodal lines along hyperbolic cusps.
There is no analogous description of nodal lines in our more general situation.

The main part of the proof of \cref{small} is concerned with topological properties
of approximate nodal domains and their asymptotic behaviour.
The analytical part of the proof of \cref{small} is concentrated in Lemma \ref{eul}.
To prepare the proof of \cref{eul},
we need some prerequisites from analysis which we present in \cref{anapre}.
In particular, we extend Cheng's Theorem 2.5 in \cite{Che}
on nodal lines of solutions of Schr\"odinger equations
to the case of surfaces with smooth boundary; see \cref{cheng} below.

Mutatis mutandis, our arguments remain valid for Schr\"odinger operators $\Delta+V$,
where the potential $V$ is non-negative or, more or less equivalently,
bounded from below.
Thus the analog of \cref{small} holds also for such operators.

\section{Prerequisites from topology}
In this section, we collect some results about the topology of surfaces.
We assume throughout that the concerned surfaces
have empty or piecewise smooth boundaries.

\begin{prop}\label{funfin}
The interior of a surface $S$ is of finite type
if and only if the fundamental group of $S$ is finitely generated. \qed
\end{prop}

Among the surfaces with boundary (possibly empty)
whose interior is of finite type,
we singled out those with compact boundary in the introduction.

For $n\ge2$, denote by $F_n$ the free group in $n$ generators
and recall that the commutator subgroup of $F_2$ is isomorphic to $F_\infty$.

\begin{prop}\label{amabel}
For a non-closed surface $S$, the following are equivalent:
\begin{compactenum}[1)]
\item
The fundamental group of $S$ is cyclic.
\item
The fundamental group of $S$ is amenable.
\item
The fundamental group of $S$ does not contain $F_2$ as a subgroup.
\item
The interior of $S$ is an open disc, annulus, or cross cap. \qed
\end{compactenum}
\end{prop}

We say that a curve in a manifold is a \emph{Jordan curve} if it is properly embedded.
Note that Jordan curves are closed as subsets of the ambient manifold.
The next assertion is Corollary A.7 in \cite{Bu2} (in the orientable case).

\begin{prop}\label{lemdisc}
Any null-homotopic Jordan loop in a surface $S$ bounds an embedded disc in $S$. \qed
\end{prop}

\begin{cor}\label{annulus}
Let $c_0$ and $c_1$ be Jordan loops in an annulus $A$
which represent the generator of the fundamental group of $A$
(up to orientation) and which do not intersect.
Then $c_0$ and $c_1$ are the boundary circles of an embedded annulus $A'$ in $A$. \qed
\end{cor}

A subsurface $C\subseteq S$ is called \emph{incompressible} in $S$
if any closed curve in $C$, which is homotopic to zero in $S$,
is already homotopic to zero in $C$.

\begin{lem}\label{sub}
Let $R$ be a compact and connected surface
(with piecewise smooth boundary $\partial R$, possibly empty)
which is not homeomorphic to the sphere.
Let $X$ be a non-empty incompressible closed subsurface of $R$
with piecewise smooth boundary $\partial X$.
Assume that $\partial X\cap\partial R$ is a union
of piecewise smooth segments and circles (possibly empty)
and that $\partial X$ and $\partial R$ are transversal, where they meet.
Then \[\chi(R)\le\chi(X).\]
In the case of equality,
the components of $R\setminus\mathring X$ are annuli, cross caps, and lunes.
More precisely, if $C$ is a component of $R\setminus\mathring X$ that intersects
the boundary of $R$,
then $C$ is an annulus attached to a boundary circle of $X$
or is a lune attached to a part of a boundary circle of $X$.
Otherwise $C$ is an annulus attached to two boundary circles of $X$
or a cross cap attached to a boundary circle of $X$.
\end{lem}

Here a \emph{lune} is a closed disc $D$ whose boundary is subdivided into two subarcs.
Attaching a lune $D$ to $X$ along $\partial X$ means to glue one of the
subarcs of the boundary of $D$ to an arc in $\partial X$.
Then $X$ is isotopic to $X\cup D$.

\begin{proof}[Proof of \cref{sub}]
We may assume $X\subsetneq R$.
Now by the assumptions on the boundaries of $R$ and $X$,
there is a closed collar $U$ about $\partial R$ in $R$
such that $Y=X\setminus\mathring U$ is a deformation retract of $X$ in $R$.
Observe that $Y$ does not intersect $\partial R$
and that the boundaries of $R$ and $Y$ each are disjoint unions of circles.

If a component $D$ of $R\setminus\mathring Y$ would be a closed disc,
set $c=\partial D$, a circle in $\partial Y$.
If $c$  would be homotopic to zero in $Y$,
then there would be a closed disc $D'$ in $Y$ with $\partial D'=c$.
Thus $D\cup D'$ would be an embedded sphere in $R$.
This is not possible since $R$ is connected and would have to be equal to that sphere. 
Thus $c$ is not homotopic to zero in $Y$, hence neither in $R$
since $Y$ is incompressible in $R$.
This is a contradiction, and hence no component of $R\setminus\mathring Y$ is a disc.
Note also that no component of $R\setminus\mathring Y$ is a closed surface
since $R$ is connected and $Y$ is non-empty.

From the Mayer-Vietoris sequence, we obtain
$\chi(R) = \chi(Y) + \chi(R\setminus\mathring Y)$.
Since no component of $R\setminus\mathring Y$ is a disc or a closed surface,
we have $\chi(R\setminus\mathring Y)\le0$ and hence
\begin{equation*}
  \chi(R) \le \chi(Y) = \chi(X).
\end{equation*}
If $\chi(R)=\chi(Y)$,
we have $\chi(C)=0$ for each component $C$ of $R\setminus\mathring Y$.
Hence each such $C$ is an annulus or a cross cap.
If $C$ does not intersect $U$, then $C$ is also a component of $R\setminus X$.

If $C$ intersects $U$, it contains the corresponding parts of the boundary of $R$.
Since $R$ is connected and $Y$ is non-empty,
$C$ also contains a part of the boundary of $Y$.
Hence the boundary of $C$ has more than one component, and hence $C$ is an annulus.
Therefore $C$ contains precisely one boundary circle of $R$
and intersects only the corresponding part of $U$.

Let $C'$ be a component of $R\setminus\mathring X$ that is contained in $C$.
If $C'$ contains a component of $\partial R$,
then $C=C'$ and $C'$ is an annulus.
If $C'$ intersects a component of $\partial R$ but does not contain it,
then $C\setminus C'$ is a subdomain of $C$
whose boundary components intersects both boundary circles of $C$.
This is possible only if $C\setminus\mathring X$ consists of attached lunes.
\end{proof}

\section{Prerequisites from analysis}\label{anapre}

We let $M$ be a Riemannian manifold, complete or not complete,
connected or not connected, with or without (piecewise smooth) boundary.
We denote by $C^k(M)$ the space of $C^k$-functions on $M$,
by $C^k_{c}(M)\subseteq C^k(M)$
the space of $C^k$-functions on $M$ with compact support,
and by $C^k_{cc}(M)\subseteq C^k_c(M)$
the space of $C^k$-functions on $M$ with compact support
in the \emph{interior $\mathring M$} of $M$, respectively.
In the case where the \emph{boundary $\partial M$} of $M$ is empty,
we have $C^k_{cc}(M)=C^k_c(M)$.
We use the term \emph{smooth} to indicate $C^\infty$.

We let $L^2(M)$ be the space of square-integrable functions on $M$
and recall that $C^\infty_{cc}(M)$ is a dense subspace of $L^2(M)$.
We denote by $H^1(M)$ the space of functions in $L^2(M)$
which have a square-integrable gradient $\nabla f$ in the sense of distributions.
By the latter we mean that we test $\nabla f$ against smooth one-forms on $M$
with compact support in $\mathring M$.
Recall that $H^1(M)$ is a Hilbert space with respect to the $H^1$-norm
and denote by $H^1_0(M)$ the closure of $C^\infty_{cc}(M)$ in $H^1(M)$.

\begin{prop}[Friedrichs extension]\label{friedex}
The Laplacian $\Delta$ is self-adjoint as an unbounded operator on $L^2(M)$
with domain the space of $\vf\in H^1_0(M)$ such that $\Delta\vf\in L^2(M)$
in the sense of distributions.
\end{prop} 

\begin{proof}
Since $H^1_0(M)$ is dense in $L^2(M)$,
this follows immediately from the construction of the Friedrichs extension;
compare with \cite[Section A.8]{T2}.
\end{proof}

We call the spectrum of the Friedrichs extension of $\Delta$ as in \cref{friedex}
the \emph{spectrum} of $M$.
Note that, in the case where $M$ has no boundary,
this is the usual spectrum of $S$.

Set $C^k_{0}(M) = \{\vf\in C^k(M)\mid\vf|_{\partial M}=0\}$,
and let $C^k_{c,0}(M)$ be the space of $C^k$-functions
in $C^k_{0}(M)$ with compact support.
We use a corresponding notation for the H\"older spaces $C^{k,\alpha}(M)$,
where $0<\alpha\le1$.

\begin{lem}\label{dirinc}
If $M$ is complete as a metric space and the boundary of $M$ is piecewise smooth
(possibly empty), then
\begin{equation*}
  C^{0,1}_{0}(M) \cap H^1(M) \subseteq H^1_0(M).
\end{equation*}
\end{lem}

\begin{proof}
Since the boundary of $M$ is piecewise smooth,
there is a sequence of functions $\chi_n$ in $C^1_c(M)$
such that $0\le\chi_n\le1$ and $|\nabla\chi_n|\le1/n$,
such that $\{\chi_n=1\}$ contains the support of $\chi_{n-1}$ in its interior,
and such that $\cup\{\chi_n=1\}=M$.
With such a sequence, we can reduce the assertion of \cref{dirinc}
to the case of functions in $C^{0,1}_{c,0}(M) \cap H^1(M)$. 

Given a compact set $K\subseteq M$,
there is a sequence of functions $\chi_n$ in $C^1_c(M)$
such that $0\le\chi_n\le1$ and $|\nabla\chi_n|\le Cn$ for some constant $C=C(K)$,
such that $\chi_n=1$ on the set of $x\in K$ with $d(x,\partial M)\ge2/n$,
and such that $\chi_n=0$ on the set of $x\in K$ with $d(x,\partial M)\le1/n$.

Now let $\vf\in C^{0,1}_{c,0}(M) \cap H^1(M)$ and $K=\supp\vf$.
Choose a sequence of functions $\chi_n$ for $K$ as above.
Then $\chi_n\vf\to\vf$ in $H^1(M)$ since the area of the set of $x\in K$
with $1/n\le d(x,\partial M)\le2/n$, which contains $K\cap\supp\nabla\chi_n$,
is bounded by $A/n$ for some constant $A$
and since $\vf\le2B/n$ on this set, where $B$ is a Lipschitz constant for $\vf$.
This reduces the assertion of \cref{dirinc} to the case where the support of $\vf$
is contained in $\mathring M$.
In this case, the assertion follows from smoothing.
\end{proof}

As in the introduction, denote by $\lambda_0(M)=\inf R(\vf)$,
where the infimum is taken over all non-zero $\vf\in C^\infty_{cc}(M)$.
Since $R$ is continuous on $H^1_0(M)\setminus\{0\}$
and $C^\infty_{cc}(M)$ is dense in $H^1_0(M)$, we have
\begin{equation}\label{rayl1}
  \lambda_0(M) = \inf \{ R(\vf) \mid \vf\in H^1_0(M)\setminus\{0\} \}. 
\end{equation}
Hence $\lambda_0(M)$ is the \emph{bottom of the spectrum} of the Laplacian.
By the definition of $\lambda_0$, we also have \emph{domain monotonicity},
\begin{equation}\label{raylm}
  \lambda_0(M) \ge \lambda_0(M')
\end{equation}
for any Riemann manifold $M'$ containing $M$.

\begin{lem}\label{rayl2}
A non-zero $\vf\in H^1_0(M)$ satisfies $R(\vf)=\lambda_0(M)$
if and only if $\vf$ is an eigenfunction of the Laplacian with eigenvalue $\lambda_0(M)$.
\end{lem}

\begin{proof}
By the spectral theorem,
we may represent $L^2(M)$ as the space $L^2(X)$ of square integrable functions
on a measured space $X$ such that $\Delta$ corresponds to multiplication
by a measurable function $f$ on $X$.
By the definiton of $\lambda_0(M)$,
we have $f\ge\lambda_0(M)\ge0$ almost everywhere on $X$.
\end{proof}

For a self-adjoint operator $A$ on a Hilbert space $H$,
the spectrum $\spec A$ of $A$ can be decomposed in several ways.
By definition, the \emph{essential spectrum} $\spec_{\rm ess}A\subseteq\spec A$
consists of all $\lambda\in\R$ such that $A-\lambda\id$ is not a Fredholm operator.
The \emph{discrete spectrum} $\spec_dA$ is the complement,
\begin{equation*}
  \spec_dA=\spec A\setminus\spec_{\rm ess}A.
\end{equation*}
The discrete spectrum consists of eigenvalues of finite multiplicity of $A$
which are isolated points of $\spec A$.
The essential spectrum is a closed subset of $\R$.

The following result shows that the essential spectrum of the Laplacian
only depends on the geometry of the underlying manifold at infinity
and that the essential spectrum of the Laplacian is empty if $M$ is compact.

\begin{prop}\label{esp}
For a complete Riemannian manifold $M$ with compact boundary (possibly empty),
$\lambda\in\R$ belongs to the essential spectrum of $\Delta$
if and only if there is a \emph{Weyl sequence} for $\lambda$, that is,
a sequence of functions $\vf_n$ in $C^\infty_{cc}(M)$ such that \\
\begin{inparaenum}[1)]
\item
for any compact $K\subseteq M$,
$\supp\vf_n\cap K=\emptyset$ for all sufficiently large $n$; \\
\item
$\limsup_{n\rightarrow\infty}\|\vf_n\|_2>0$
and $\lim_{n\rightarrow\infty}\|\Delta\vf_n-\lambda\vf_n\|_2=0$.
\end{inparaenum}
\end{prop}

\begin{proof}
See the elementary argument in the proof of Proposition 1 in \cite{Bae}.
\end{proof}

\begin{exas}\label{exesp}
\begin{inparaenum}[1)]
\item\label{many}
Let $S$ be a non-compact hyperbolic surface without boundary and with finite area.
Replace a simple closed geodesic $c$ in $S$ by a Euclidean cylinder
$C=\{(x,y)\mid 0\le x\le h,\, y\in\R/L\Z\}$
of height $h$ and circumference $L=L(c)$
and smooth out the resulting Riemannian metric appropriately.
Let $\varphi=\varphi(x)$ be a non-vanishing smooth function on $\R$
with support in $[-1,0]$.
Then the support of $\varphi_{k,i}=\varphi(x/k-i)$ is in $[(i-1)k,ik]$
and the Rayleigh quotient of $\varphi_{k,i}$ is $R(\varphi)/k^2$.
Hence, given $\ve>0$, we have $R(\vf_{k,i})<\ve$ if $k^2>R(\varphi)/\ve$.
We may also view $\varphi_{k,i}$ as a smooth function on the cylinder $C$
and the surface $S$ if $h$ is sufficiently large.
More specifically, given $n$, choose $h>nk$.
Then the functions $\varphi_{k,1},\dots,\varphi_{k,n}$ have disjoint supports
in $C$ and Rayleigh quotients $<\ve$.
Hence $S$ has at least $n$ eigenvalues which are $<\ve$.
Since $C$ is a cylinder, we also have $\Lambda(S)<\ve$.
On the other hand,
the essential spectrum of $S$ is still contained in $[1/4,\infty)$,
by \cref{esp}.

\item\label{emesp}
Let $F=\{(x,y)\mid x\ge0,\, y\in\R/L\Z\}$ be a funnel with the expanding
hyperbolic metric $dx^2+\cosh(x)^2 dy^2$.
Let $\kappa\colon\R\to\R$ be a monotonic smooth function with $\kappa(x)=-1$
for $x\le1$ and $\kappa(x)\to-\infty$ as $x\to\infty$.
Suppose that $j\colon\R\to\R$ solves $j''+\kappa j=0$
with initial condition $j(0)=1$ and $j'(0)=0$.
Then $j(x)>\cosh x$ for all $x>1$.
Furthermore, the funnel $F$ with Riemannian metric $g=dx^2+j(x)^2dy^2$
has curvature $K(x,y)=\kappa(x)\le-1$ and infinite area.
By comparison, the Rayleigh quotient with respect to $g$
of any smooth function $\vf$ with compact support in the part $\{x\ge x_0\}$
of the funnel is at least $-\kappa(x_0)/4$.

Let $S$ be a non-compact surface of finite type.
Endow $S$ with a hyperbolic metric which is expanding along its funnels as above.
Replace the hyperbolic metric on the funnels by the above Riemannian metric $g$.
Then the new Riemannian metric on $S$ is complete
and has curvature $K\le-1$ and infinite area.
By \cref{esp} and by what we said above about the Rayleigh quotients,
the essential spectrum of the new Riemannian metric is empty.

\item\label{emesp2}
As a variation of \ref{emesp}),
suppose now that $j$ is the unique solution of $j''+\kappa j=0$
which satisfies the boundary condition $j(0)=1$ and $j(\infty)=0$.
Then $j'(0)<-1$ and $j(x)<\exp(-x)$ for all $x>0$.
The funnel $F$ with Riemannian metric $g=dx^2+j(x)^2dy^2$
has curvature $K(x,y)=\kappa(x)$ and finite area.
Again by comparison, the Rayleigh quotient with respect to $g$
of any smooth function $\vf$ with compact support in the part $\{x\ge x_0\}$
of the funnel is at least $-\kappa(x_0)/4$.

Let $S$ be a non-compact surface of finite type,
and choose $r>0$ such that $\coth(r)=-j'(0)$.
It is not hard to see that $S$ minus the parts $\{x\ge r\}$ of its funnels
carries hyperbolic metrics which are equal to $dx^2+j_0(x)^2dy^2$
along the parts $\{x<r\}$ of its funnels, where $j_0(x)=\sinh(r-x)/\sinh(r)$.
Then $j_0(x)=j(x)$ for $x<\min\{1,r\}$.
Hence any such hyperbolic metric,
restricted to $S$ minus the parts $\{x\ge\min\{1,r\}\}$ of its funnels,
when combined with $g$ along the funnels,
defines a smooth and complete Riemannian metric on $S$
which has curvature $K\le-1$ and finite area.
Again, its essential spectrum is empty,
by \cref{esp} and by what we said above about the Rayleigh quotients.
\end{inparaenum}
\end{exas}

Although we will not need the following consequence of \cref{esp} here,
we state it for general reference.
For a complete Riemannian manifold $M$  with compact boundary (possibly empty),
we denote by $\lambda_{\rm ess}(M)$ the bottom of the essential spectrum of $M$.

\begin{cor}\label{esp2}
Let $M$ be a complete Riemannian manifold with compact boundary (possibly empty)
and finitely many ends.
Assume that $M$ admits a neighborhood $U$ of infinity such that,
for any connected component $C$ of $U$,
the image of $\pi_1(C)$ in $\pi_1(M)$ is amenable.
Then $\lambda_{\rm ess}(M)\ge\lambda_0(\tilde M)$.
\end{cor}

\begin{proof}
Let $\lambda\in\spec_{\rm ess}\Delta$ and $(\vf_n)$
be a Weyl sequence for $\lambda$ as in \cref{esp}.
Then, by passing to a subsequence if necessary,
we can assume that all $\vf_n$ have support in a connected component $C$
of a neighborhood $U$ of infinity of $M$ as in the assumption.
Then the inclusion $C\to M$ and, with it, all $\vf_n$ can be lifted
to the subcovering $\bar M=\Gamma\backslash\tilde M$
of the universal covering space $\tilde M$ of $M$,
where $\Gamma$ denotes the image of $\pi_1(C)$ in $\pi_1(M)$.
Therefore we have $\lambda\ge\lambda_0(\bar M)$.
Now $\Gamma$ is amenable
and hence $\lambda_0(\bar M)=\lambda_0(\tilde M)$,
by Theorem 1 of \cite{Br2} (extended to manifolds with boundary).
\end{proof}

\begin{rem}
For complete Riemannian surfaces $S$ of finite type with compact boundary,
we have the refinement $\lambda_{\rm ess}(S)\ge\Lambda(S)\ge\lambda_0(\tilde S)$;
see \cref{rems}.\ref{broo} and \ref{rems}.\ref{ess} in the introduction.
\end{rem}

In the case of surfaces without boundary,
the next result is Theorem 2.5 in \cite{Che}.

\begin{thm}\label{cheng}
Let $S$ be a surface with smooth boundary (possibly empty),
endowed with a Riemannian metric.
Let $\vf,V$ be smooth functions on $S$
and suppose that $\vf$ vanishes along the boundary of $S$
and solves the Schr\"odinger equation $(\Delta+V)\vf=0$.
Then the \emph{nodal set} $Z_\vf=\{x\in S\mid \vf(x)=0\}$
of $\vf$ is a locally finite graph in $S$.
Moreover,
\begin{compactenum}[1)]
\item
$z\in Z_\vf\cap\mathring S$ has valence $2n$
if and only if $\vf$ vanishes of order $n$ at $z$.
\item
$z\in Z_\vf\cap\partial S$ has valence $n+1$
if and only if $\vf$ vanishes of order $n$ at $z$.
\end{compactenum}
In both cases, the opening angles between the edges at $z$ are equal to $\pi/n$.
\end{thm}

\begin{proof}
Recall that non-zero eigenfunctions of the Laplacian
cannot vanish of infinite order at any point; see e.g. \cite{Ar}.
Hence by the main result of \cite{Be}, at any critical point $z\in Z_\vf\cap\mathring S$ of $\vf$,
there are Riemannian normal coordinates $(x,y)$ about $z$,
a spherical harmonic $p=p(x,y)\ne0$ of some order $n\ge2$,
and a constant $\alpha\in(0,1)$ such that
\begin{equation*}
   \vf(x,y) = p(x,y) + O(r^{n+\alpha}),
\end{equation*}
where we write $(x,y)=(r\cos\theta,r\sin\theta)$.
By Lemma 2.4 of \cite{Che}, there is a local $C^1$-diffeomorphism $\Phi$
about $0\in\R^2$ fixing $0$ such that $\vf=p\circ \Phi$.
Note that, up to a rotation of the $(x,y)$-plane, we have
\begin{equation*}
  p=p(x,y)=cr^n\cos n\theta
\end{equation*}
for some constant $c\ne0$.
It follows that the interior nodal set $Z_\vf\cap\mathring S$ of $\vf$
is a locally finite graph with critical points of $\vf$ as vertices
and that the valence of points on $Z_\vf$ is as asserted.

It remains to discuss points $z\in Z_\vf\cap\partial S$.
Since $\dim S=2$, there are isothermal coordinates around $z$, that is,
coordinates $(x,y)$ about $z$
in which the Riemannian metric $g$ of $S$ is conformal to the Euclidean metric $g_0$:
$g=fg_0$ with $f=f(x,y)>0$.
Then, again since $\dim S=2$, the associated Laplacians satisfy $f\Delta=\Delta_0$,
and hence $\vf$ solves  the Schr\"odinger equation $(\Delta_0+fV)\vf=0$
in the domain of the coordinates.

After an appropriate further conformal change of the coordinates,
we can assume that the domain of the coordinates is $B_\ve(0)\cap\{y\ge0\}$
such that $\partial S$ corresponds to $B_\ve(0)\cap\{y=0\}$.
We consider $\vf$ and $W=fV$ as functions on $B^+=B_\ve(0)\cap\{y\ge0\}$,
where $\vf(x,0)=0$, and extend them to functions on $B_\ve(0)$
by setting $\vf(x,y)=-\vf(x,-y)$ and $W(x,-y)=W(x,y)$.
Then $\vf$ and $W$ are $C^{1,1}$ and $C^{0,1}$ on $B_\ve(0)$, respectively,
and $\vf$ solves $(\Delta_0+W)\vf=0$ in $B^+$.
Since the reflection about the $x$-axis is an isometry of the Euclidean plane,
we also have $(\Delta_0\vf)(x,y)=-(\Delta_0\vf)(x,-y)$.
Hence
\begin{equation*}
  (\Delta_0+W)\vf(x,y) = -(\Delta_0\vf)(x,-y) - W(x,y)\vf(x,-y) = 0
\end{equation*}
in $B^-=B_\ve(0)\cap\{y\le0\}$.
Since $\vf=0$ along the $x$-axis, all $x$-derivatives of $\vf$ vanish along the $x$-axis.
Since $\vf$ solves $(\Delta_0+W)\vf=0$,
the second derivative of $\vf$ in the $y$-direction vanishes along the $x$-axis as well,
and hence $\vf$ is $C^{2,1}$.
We conclude that $\vf$ is a strong solution of $(\Delta_0+W)\vf=0$ on $B_\ve(0)$,
and hence the main result of \cite{Be} and (the proof of)  Lemma 2.4 of \cite{Che} applies.
The remaining assertions follow as in the case of $z\in Z_\vf\cap\mathring S$ above.
\end{proof}

We learned from the proof of Theorem 2.3 in \cite{HHT} that the reflection
about the $x$-axis in the Euclidean plane,
which we use in the second part of the above proof,
might be helpful in the discussion of the boundary regularity
of solutions of Schr\"odinger equations.

\begin{cor}\label{cheng2}
In the situation of \cref{cheng},
$Z_\vf$ is a locally finite union of immersed circles,
line segments with both end points on $\partial S$,
rays with one end point on $\partial S$, and lines. \qed
\end{cor}

\section{Proof of \cref{small}}
Throughout this section,
let $\vf$ be a non-vanishing square integrable smooth function on $S$
which is a finite linear combination of eigenfunctions
with eigenvalues $\le\Lambda(S)$.
The set of zeros of $\vf$,
\begin{equation}\label{nodal}
  Z_\vf := \{x\in S\mid \vf(x)=0\},
\end{equation}
is called the {\it nodal set} of $\vf$.
The connected components of the complement $S\setminus Z_\vf$
are called \emph{nodal domains} of $\vf$.

\begin{lem}\label{lemnae}
With respect to the area element of $S$,
we have $\nabla\vf(x)=0$ for almost any $x\in Z_\vf$.
\end{lem}

\begin{proof}
With respect to the area element of $S$,
the set of points of density of $Z_\vf$ has full measure in $Z_\vf$.
Clearly, $\nabla\vf(x)=0$ in any such point $x$.
\end{proof}

We say that $\ve>0$ is \emph{$\vf$-regular},
if $\ve$ and $-\ve$ are regular values of $\vf$.
For any $\ve>0$, we call
\begin{equation}\label{defzef}
  Z_\vf(\ve) := \{x \in S\mid |\vf(x)| \leq \ve \}
\end{equation}
the \emph{$\ve$-nodal set} of $\vf$.
We are only interested in the case where $\ve$ is $\vf$-regular.
Then ${Z_\vf}(\ve)$ is a subsurface of $S$ with smooth boundary,
may consist of more than one component,
and the boundary components of ${Z_\vf}(\ve)$ are embedded smooth circles
and lines along which $\vf$ is constant $\pm\ve$.

\begin{lem}\label{lemfe}
For any $\vf$-regular $\ve>0$, consider the function $\vf_\ve$
\begin{equation*}
  \vf_\ve(x) = \begin{cases}
  \vf(x)-\ve \:&\text{if $\vf(x)\ge\ve$}, \\
  \vf(x)+\ve &\text{if $\vf(x)\le-\ve$}, \\
    \phantom{\vf(x)} 0 &\text{otherwise}.
 \end{cases}
\end{equation*}
Then $\vf_\ve\in H^1(M)$ and $\lim_{\ve\rightarrow0}\vf_\ve=\vf$ in $H^1(M)$.
\end{lem}

\begin{proof}
For all $x\in S$, we have $|\vf_\ve(x)|\le|\vf(x)|$.
Hence $\vf_\ve$ is in $L^2(M)$. Moreover, $\vf_\ve(x)\rightarrow\vf(x)$,
hence $\lim_{\ve\rightarrow0}\vf_\ve=\vf$ in $L^2(M)$.
Furthermore, $\vf_\ve$ has weak gradient
\begin{equation}\label{weakder}
  \nabla\vf_\ve(x) = \begin{cases}
  \nabla\vf(x) \:&\text{if $|\vf(x)|\ge\ve$}, \\
  \phantom{\nabla\vf} 0 &\text{otherwise}.
 \end{cases}
\end{equation}
It follows that $\vf_\ve$ is in $H^1(M)$.
By \cref{lemnae},
$\lim_{\ve\rightarrow0}\vf_\ve=\vf$ in $H^1(M)$.
\end{proof}

In what follows, we assume throughout that $\ve$ is $\vf$-regular.
We say that a disc $D$ in $S$ is an \emph{$\ve$-disc} if $D$ is closed in $S$ and
\begin{equation}\label{norder}
  \text{$\vf=+\ve$ and $\nu(\vf)>0\quad$ or $\quad\vf=-\ve$ and $\nu(\vf)<0$}
\end{equation}
along the boundary circle $\partial D$ of $D$,
where $\nu$ denotes the outer normal of $D$ along $\partial D$.
Note that, for an $\ve$-disc $D$,
a neighborhood of $\partial D$ inside $D$ is contained in $Z_\vf(\ve)$,
whereas a neighborhood of $\partial D$ outside $D$ belongs to $\{\vf\ge\ve\}$
in the first case in \eqref{norder} and $\{\vf\le-\ve\}$ in the second.

The boundary circles of $\ve$-discs are components of $\{\vf=\pm\ve\}$.
Since $\ve$ is $\vf$-regular,
the normal derivative of $\vf$ has to be nonzero along $\{\vf=\pm\ve\}$.
The requirements on the normal derivative in \eqref{norder} fix its sign.
As an example where these requirements do not hold,
we note that components of $\{\vf\ge\ve\}$ or $\{\vf\le-\ve\}$ might be discs,
but never $\ve$-discs.
On the other hand, any component of $Z_\vf(\ve)$, which is a disc, is also an $\ve$-disc.

By the Schoenflies theorem, any component $C$ of $Z_\vf(\ve)$,
which is contained in the interior of a closed disc, is also contained in an $\ve$-disc.
More precisely, there is an $\ve$-disc $D$ such that $\partial D\subseteq\partial C$
and such that $C$ is a neighborhood of $\partial D$ inside $D$.
We let $Y_\vf(\ve)$ be the union of $S\setminus Z_\vf(\ve)$ with all $\ve$-discs.
Note that the union might not be disjoined since $\ve$-discs might contain components
of $\{\vf\ge\ve\}$ and $\{\vf\le-\ve\}$.

\begin{lem}\label{yfi}
\begin{inparaenum}[1)]
\item\label{zdisc}
$Y_\vf(\ve)$ is the union of $S\setminus\mathring Z_\vf(\ve)$
with all components of $Z_\vf(\ve)$
which are contained in the interior of closed discs in $S$. \\
\item\label{cinco}
The components of $Y_\vf(\ve)$ are incompressible in $S$.
\end{inparaenum} 
\end{lem}

\begin{proof}
\ref{zdisc}) follows immediately from what we said above.
As for \ref{cinco}), suppose that there is a loop in a component $C$ of $Y_\vf(\ve)$
which is not contractible in $C$, but is contractible in $S$.
Then there is an embedded circle $c$ in the interior of $C$ with that property.
By \cref{lemdisc}, there is a closed disc $D$ in $S$ with $\partial D=c$.
Since $c$ is not contractible in $C$,
the interior of $D$ contains a component of $S\setminus Y_\vf(\ve)\subseteq Z_\vf(\ve)$.
This contradicts \ref{zdisc}).
\end{proof}

The set of $\ve$-discs is ordered by inclusion.
It is important that we have maximal elements in this ordered set.

\begin{lem}\label{maxed}
If two $\ve$-discs intersect, then they are either identical
or one is contained in the interior of the other.
Moreover, any $\ve$-disc is contained in a unique maximal $\ve$-disc
and maximal $\ve$-discs are either identical or disjoint.
\end{lem}

\begin{proof}
The first statement is clear since $\ve$ is $\vf$-regular.

Fix an exhaustion of $S$ by compact subsurfaces $S_n$
such that $S\setminus S_n$ consists of cylindrical neighborhoods of the ends of $S$
and such that $\partial S_n$ intersects the set $\{\vf=\pm\ve\}$ transversally.
Then the boundary components of any $S_n$ are labeled by the ends of $S$ they belong to,
any $S_n$ meets only finitely many components of the set $\{\vf=\pm\ve\}$,
and the sets $\{\vf=\pm\ve\}\cap\partial S_n$ are finite.

Let $D_1\subseteq D_2\subseteq\dots$ be an ascending chain of pairwise distinct $\ve$-discs.
For $n$ sufficiently large, we have $D_1\subseteq S_n$. 
Then $D_l\cap S_n\ne\emptyset$
and $\partial D_l\cap\partial S_n\subseteq\{\vf=\pm\ve\}\cap\partial S_n$ for all $l\ge1$.
Moreover, if $\partial D_l\cap\partial S_n=\emptyset$, then $D_l\subseteq S_n$.
Since $\ve$ is $\vf$-regular, it follows that the chain of discs is finite.

Let $D$ and $D'$ be maximal $\ve$-discs and suppose that $D\cap D'\ne\emptyset$.
Note that $D$ and $D'$ each have only one boundary circle, $c$ and $c'$.
If $c=c'$, then $D=D'$ by maximality.
If $c$ is contained in the interior of $D'$, then $c'$ is contained in the interior of $D$,
since otherwise $D$ would be contained in the interior of $D'$, contradicting maximality.
But then $D\cup D'$ is a subsurface of $S$ without boundary which is closed as a subset,
and hence $D\cup D'=S$.
This is impossible since then $S=D\cup D'$ would be a sphere. 
\end{proof}

\begin{lem}\label{sinc}
Each component $C$ of $Y_\vf(\ve)$ is the union of a component $C_0$
of $\{\vf\ge+\ve\}$ or of $\{\vf\le-\ve\}$
together with maximal $\ve$-discs attached to them along common boundary circles.
In particular, $\partial C\subseteq\partial C_0$ and
\begin{align*}
  &\text{$\vf|_{\partial C}=+\ve$ and $\nu(\vf)<0$ if $C_0$ is contained in $\{\vf\ge+\ve\}$,} \\
  &\text{$\vf|_{\partial C}=-\ve$ and $\nu(\vf)>0$ if $C_0$ is contained in $\{\vf\le-\ve\}$,}
\end{align*}
where $\nu$ denotes the outer normal field of $C$.
\end{lem}

\begin{proof}
By \cref{maxed},
$Y_\vf(\ve)$ is the union of $S\setminus Z_\vf(\ve)$ with all maximal $\ve$-discs.
By the requirement on the normal derivative in \eqref{norder},
the boundary circle $c$ of each maximal $\ve$-disc $D$
is attached to a component of $\{\vf\ge\ve\}$ if $\vf|_c=\ve$
and a component of $\{\vf\le-\ve\}$ if $\vf|_c=-\ve$, respectively.
\end{proof}

By \cref{sinc}, we may write $Y_\vf(\ve)$ as the disjoint union,
\begin{equation}\label{ypm}
  Y_\vf(\ve) = Y_\vf^+(\ve) \dot\cup Y_\vf^-(\ve),
\end{equation}
where $Y_\vf^+(\ve)$ and $Y_\vf^-(\ve)$ consist of the components $C$ of $Y_\vf(\ve)$
such that the corresponding $C_0$  is contained in $\{\vf\ge+\ve\}$ and $\{\vf\le-\ve\}$,
respectively.

For the statement of the following lemma, recall \cref{amabel}.

\begin{lem}\label{eul}
For any sufficiently small $\vf$-regular $\ve>0$,
the fundamental group of at least one component of $Y_\vf(\ve)$ contains
the free group $F_2$ in two generators.
Moreover, if $\vf$ is an eigenfunction,
then each nodal domain $C$ of $\vf$ is incompressible
and the fundamental group of $C$ contains $F_2$. 
\end{lem}

\begin{proof}
We may assume that $Y_\vf(\ve)\ne S$.
We suppose first that the Rayleigh quotient $R(\vf)<\Lambda(S)$
and choose $\delta > 0$ such that
\begin{equation}\label{strict}
  R(\vf) \le \Lambda(S) - 2\delta.
\end{equation}
By \cref{lemfe} and since $S\setminus Y_\vf(\ve)\subseteq Z_\vf(\ve)$, we have,
for any sufficiently small $\vf$-regular $\ve>0$,
\begin{equation}\label{strict2}
   \frac{\sum_C\int_C|\nabla\vf_\ve|^2}{\sum_C\int_C\vf_\ve^2}
   \le \frac{\int_S|\nabla\vf|^2 dv}{\int_S\vf^2 dv} + \delta
   = R(\vf) + \delta
   \leq \Lambda(S) - \delta,
\end{equation}
where the sums run over the components $C$ of $Y_\vf(\ve)$.
We conclude that there is a component $C$ of $Y_\vf(\ve)$ such that
\begin{equation}\label{strict4}
   R(\vf_\ve|_C)
   = \frac{\int_C|\nabla\vf_\ve|^2}{\int_C\vf_\ve^2}
   \leq \Lambda(S) - \delta.
\end{equation}
Now $\vf$ is smooth on $S$,
hence $\vf_\ve|_C$ is smooth on $C$ and vanishes along $\partial C$.
Therefore $\vf_\ve|_C\in H^1_0(C)$, by \cref{dirinc}.
Now it follows from the definition of $\Lambda(S)$ that the interior of $C$
can not be diffeomorphic to an open disc, an open annulus, or an open cross cap.
Thus the fundamental group of $C$ contains $F_2$.

Assume now that $R(\vf)=\Lambda(S)$.
Recall that $\vf$ is a finite linear combination of eigenfunctions of $S$,
$\vf = \sum  c_i\vf_i$,
where $\vf_i\in E$ is a $\lambda_i$-eigenfunction with $\lambda_i\le\Lambda(S)$.
If there would be an $i$ with $c_i\ne0$ and $\lambda_i<\Lambda(S)$,
then we would have $R(\vf)<\Lambda(S)$, a contradiction.
It follows that all $\lambda_i$ with $c_i\ne0$ are equal to $\Lambda(S)$,
and hence that $\vf$ is a $\Lambda(S)$-eigenfunction.

Suppose now, more generally, that $\vf$ is an eigenfunction
with corresponding eigenvalue $\lambda\le\Lambda(S)$.
Then $\vf$ is smooth on $S$.
By \cref{cheng}, the nodal domains of $\vf$ have piecewise smooth boundary.
Hence \cref{dirinc} implies that, for any nodal domain $C$ of $\vf$,
we have $\vf|_C\in H^1_0(C)$ with $R(\vf|_C)=\lambda$.
In particular, $\lambda_0(C)\le\lambda$.

Let $C'$ be a thickening of $C$,
that is, $C'$ is a domain in $S$ with piecewise smooth boundary
which contains $C$ in its interior and such that $C$ is a deformation retract of $C'$.
If the fundamental group of $C$ does not contain $F_2$,
then neither does the fundamental group of $C'$, and then
\begin{equation*}
  \Lambda(S) \le \lambda_0(C') \le \lambda_0(C).
\end{equation*}
The extension $\vf'$ of $\vf|_{\partial C}$ to $C'$, setting $\vf'|_{C'\setminus C}=0$,
is in $H^1_0(M)$ and in the domain of the Laplacian of $C'$.
Moreover, it has Rayleigh quotient $R(\vf')=\lambda$.
Hence \cref{rayl2} applies and shows that $\lambda_0(C')=\Lambda(S)$
and $\Delta\vf'=\Lambda(S)\vf'$.
Now $\vf'$ does not vanish identically on $C$, but vanishes on $C'\setminus C$.
This is in contradiction to the unique continuation property for Laplace operators.
Hence the fundamental group of any nodal domain of $\vf$ contains $F_2$.

Let $C$ be a nodal domain of $\vf$ and suppose that $C$ is not incompressible in $S$.
Then there is a loop $c$ in $C$ which is not homotopic to zero in $C$,
but is homotopic to zero in $S$.
Without loss of generality, we may assume that $c$ is a Jordan curve in $\mathring C$.
Then $c$ bounds a disc in $S$, which is not contained in $C$, by \cref{lemdisc}.
By the Schoenflies theorem, there would be a nodal domain $D$ of $\vf$
whose closure is a closed disc with piecewise smooth boundary
and with $\lambda_0(D)=\Lambda(S)$.
This is impossible, since $\Lambda(S)$ is not attained on (embedded) closed discs.
Hence all nodal domains of $\vf$ are incompressible in $S$.

Let $C$ be a nodal domain and $c_1,c_2\colon[0,1]\to C$
be two loops at a point $x\in C$ which generate a free subgroup $F_2\in\pi_1(C,x)$.
By \cref{cheng},
we may assume that the images of $c_1$ and $c_2$ are contained in $\mathring C$.
Without loss of generality, we may also assume that $\vf$ is positive on $\mathring C$.
Then
\begin{equation*}
  \mathring C=\cup_{\ve>0}\{y\in C\mid\vf(y)>\ve\}.
\end{equation*}
Therefore the image of $c_0$ and $c_1$ is contained in $\{y\in C\mid\vf(y)>\ve\}$
for all sufficiently small $\ve>0$.
Hence the fundamental group of the component $C_0\subseteq C$ of $Y_\vf(\ve)$
which contains $c_0$ and $c_1$ also contains $F_2$.
From \cref{sinc} we conclude that the component of $Y_\vf(\ve)$ containing $C_0$
contains $F_2$.
\end{proof}

Say that a component $C$ of $Y_\vf(\ve)$ is an \emph{$F_2$-component}
if the fundamental group of $C$ contains $F_2$.
By \cref{sinc}, a component $C$ of $Y_\vf^+(\ve)$ is an $F_2$-component
if and only if there is a point $x$ in the interior of the corresponding component $C_0$
of $\{\vf\ge\ve\}$ and a pair of loops $c_0$ and $c_1$ at $x$
which generate an $F_2$ in $\pi_1(C,x)$
such that $c_0$ and $c_1$ are contained in the interior of $C_0$;
that is, such that $\vf>\ve$ along them.
The characterization of $F_2$-components of $Y_\vf^-(\ve)$ is analogous.

\begin{lem}\label{finite}
If $K$ is a compact subsurface of $S$ such that $S\setminus K$ is a cylindrical
neighborhood of the ends of $S$ and $C$ is a component of $Y_\vf(\ve)$
which is contained in $S\setminus K$,
then the interior of $C$ is diffeomorphic to an open disc or an open annulus.
In particular,
\begin{compactenum}[1)]
\item
any $F_2$-component of $Y_\vf(\ve)$ intersects $K$;
\item
$Y_\vf(\ve)$ contains only finitely many $F_2$-components.
\end{compactenum}
\end{lem}

Components of $Y_\vf(\ve)$ might be non-compact
and might have infinitely many boundary components.
But since they are incompressible in $S$, their interiors are surfaces of finite type.

\begin{proof}[Proof of \cref{finite}]
The components of $S\setminus K$ are diffeomorphic to open annuli,
hence their fundamental group is infinite cyclic.
Moreover, $C$ is incompressible in $S$,
hence also in the component of $S\setminus K$ containing it.
Therefore the fundamental group of $C$ is either trivial or infinite cyclic.
Now, as a domain in a cylinder, $C$ is orientable.
Hence the interior of $C$ is diffeomorphic to an open disc or an open annulus.
\end{proof}

We denote by $X_\vf(\ve)$ the union of the $F_2$-components of $Y_\vf(\ve)$
and by $X_\vf^\pm(\ve)$ the ones among them which belong to $Y_\vf^\pm(\ve)$.
Then $X_\vf(\ve)$ is the disjoint union of $X_\vf^+(\ve)$ and $X_\vf^-(\ve)$.

\begin{lem}\label{euler0}
If $\ve'\le\ve$ are $\vf$-regular, then $X_\vf^\pm(\ve)\subseteq X_\vf^\pm(\ve')$.
\end{lem}

\begin{proof}
By definition, $Z_\vf(\ve')\subseteq Z_\vf(\ve)$.
If a component of $Z_\vf(\ve)$
is contained in the interior of an embedded closed disc,
then also all the components of $Z_\vf(\ve')$ it contains.
It follows that $Y_\vf(\ve)\subseteq Y_\vf(\ve')$.
By \cref{yfi}.\ref{cinco}, if $C$ is an $F_2$-component of $X_\vf^\pm(\ve)$,
then also the component $C'$ of $X_\vf^\pm(\ve')$ containing it. 
Hence $X_\vf(\ve)\subseteq X_\vf(\ve')$.

Now let $C$ be a component of $X_\vf^+(\ve)$
and suppose that the component $C'$ of $X_\vf(\ve')$ 
containing $C$ belongs to $X_\vf^-(\ve')$.
By \cref{sinc}, $C$ is the union of a component $C_0$ of $\{\vf\ge\ve\}$
with maximal $\ve$-discs.
Now let $x$ be a point in the interior of $C_0$ and choose loops $c_0$ and $c_1$
in the interior of $C_0$ generating an $F_2$ in $\pi_1(C,x)$;
compare with our discussion further up.
In particular, $\vf>\ve$ along $c_0$ and $c_1$.
Under the inclusion $C\to C'$,
$c_0$ and $c_1$ cannot be contained in the maximal $\ve'$-discs belonging to $C'$
because they would be homotopic to zero in $S$ otherwise.
But then they must meet $\{\vf\le-\ve'\}$, a contradiction.
We conclude that $X_\vf^+(\ve)\subseteq X_\vf^+(\ve')$;
similarly $X_\vf^-(\ve)\subseteq X_\vf^-(\ve')$.
\end{proof}

We view the funnels of $S$ as vertical and pointing upwards.
In this picture, a Jordan curve $c$ in a funnel $F$,
which is a generator of the fundamental group of $F$,
cuts $S\setminus c$ in two open pieces,
the set $F_c$ of \emph{points above $c$} and the set of remaining points,
sometimes called the \emph{points below $c$}.
The set of points above $c$ is contained in $F$ and is a funnel around the same end as $F$,
the set of points below $c$ is not contained in $F$.

We call Jordan curves in $F$, which generate the fundamental group of $F$,
\emph{cross sections of $F$}.
We say that a cross section $c$ of $F$ is $(\vf,\ve)$-regular
if it meets the curves $\{\vf=\pm\ve\}$ transversally.
By transversality theory, any cross section of $F$ can be approximated
by smooth $(\vf,\ve)$-regular cross sections of $F$ in any reasonable topology.

Our aim is now to describe the structure of $X_\vf^\pm(\ve)$ with respect to $F$.
Let $c$ be a $(\vf,\ve)$-regular cross section of $F$.
Then $c$ intersects $\{\vf=\ve\}$ transversally.
We emphasize the following three cases:
\begin{compactenum}[1)]
\item\label{easy1}
$F_c\subseteq X_\vf^\pm(\ve)$.
\item\label{easy2}
$c\cap X_\vf^\pm(\ve)=\emptyset$,
\item\label{strips}
$c\cap\partial X_\vf^\pm(\ve)\ne\emptyset$.
\end{compactenum}
We now want to normalize the position of a $(\vf,\ve)$-regular cross section $c$ of $F$
in such a way that the part of $X_\vf^\pm(\ve)$ below $c$
is homotopy equivalent to $X_\vf^\pm(\ve)$.
If it is possible to choose $c$ such that \ref{easy1}) or \ref{easy2}) hold,
then any such choice will be a normalization.
In the remaining case, $c\cap\partial X_\vf^\pm(\ve)\ne\emptyset$
for any choice of cross section of $F$.
Since $\ve$ is $\vf$-regular, $c\cap\{\vf=\pm \ve\}$ is finite.
By \cref{sinc}, $c\cap\partial X_\vf^\pm(\ve)\subseteq c\cap\{\vf=\pm \ve\}$.

Since $\{\vf=\pm\ve\}$ is a properly embedded submanifold (of dimension $1$) of $S$,
the components of $F_c\cap\{\vf=\pm\ve\}$ above $c$ are of the following two types:
Either they are Jordan segments with endpoints on $c$ or they are Jordan rays with
one end on $c$ and escaping to infinity along the other.
We call these components \emph{recurrent} and \emph{escaping}, respectively.
Since $\{\vf=\pm\ve\}$ is properly embedded, escaping components in $F_c\cap\{\vf=\pm\ve\}$ 
extend continuously as Jordan curves to the one point compactifictaion of $F_c$ at infinity.

If $a$ is a recurrent component, then there is a segment $b$ in $c$ such that
$a\cup b$ is a null homotopic Jordan loop in $F$.
The disc bounded by $a\cup b$ will be called the part of $F_c$ below $a$.
Since $c\cap\{\vf=\pm\ve\}$ is finite, there are only finitely many such discs,
and they are ordered by inclusion.
The components $a$ above maximal such discs will be called \emph{uppermost}.
We replace the segments $b$ of $c$ below such maximal discs by the
corresponding uppermost components $a$
and obtain a piecewise smooth cross section of $F$.
Pushing this cross section upwards and smoothing it appropriately,
we arrive at the normalized third case:
$c$ is $(\vf,\ve)$-regular and the interior of $F_c\cap X_\vf^\pm(\ve)$
is a finite union of open discs,
bounded by segments of $c$, escaping components of $F_c\cap \partial X_\vf^+(\ve)$,
and, possibly, boundary lines of $X_\vf^\pm(\ve)$
which start and end at infinity in $F$.
Note that boundary circles of $X_\vf^\pm(\ve)$ cannot occur,
since they would not be null homotopic and we would be in the second case above.

In all three cases, after normalization,
the part of $X_\vf^\pm(\ve)$ below $c$ is homotopy equivalent to $X_\vf^\pm(\ve)$.
With a bit of more work, it would be possible to show that the part of $X_\vf^\pm(\ve)$
below $c$ is a deformation retract of $X_\vf^\pm(\ve)$.
The technical problem consists in handling the components of $F_c\cap \partial X_\vf^+(\ve)$
above $c$ which contain boundary lines which come from and return back to infinity in $F_c$.
These boundary lines cut out \emph{infinite peninsulas} which are hanging down from infinity
in our picture of $F_c$.
Since we do not need more than homotopy equivalence, we leave it with these remarks.

Consider a pair $(\ve,K)$,
where $\ve>0$ and $K$ is a smooth and compact subdomain of $S$
such that $S\setminus K$ consists of funnels.
Say that the pair $(\ve,K)$ is \emph{$\vf$-regular} if $\ve$ is $\vf$-regular
and $\partial K$ consists of normalized $(\vf,\ve)$-regular cross sections as above.
For any such pair $(\ve,K)$, define
\begin{equation}\label{xvek}
  X_\vf(\ve,K) = X_\vf(\ve) \cap K
  \quad\text{and}\quad
  X_\vf^\pm(\ve,K) = X_\vf^\pm(\ve) \cap K.
\end{equation}
Note that $X_\vf(\ve,K)$ is the disjoint union $X_\vf^+(\ve,K)\dot\cup X_\vf^-(\ve,K)$.

By what we said above,
the inclusions $X_\vf^\pm(\ve,K)\to X_\vf^\pm(\ve)$ are homotopy equivalences.
Since $K$ is a deformation retraction of $S$,
 \cref{sub} implies
\begin{equation}\label{eular}
  \chi(S) = \chi(K) \le \chi(X_\vf(\ve,K)) = \chi(X_\vf^+(\ve,K)) + \chi(X_\vf^-(\ve,K)) < 0.
\end{equation}
By \cref{yfi}.\ref{cinco},
the components of $X_\vf(\ve,K)$ are incompressible in $S$.
The following result is an immediate consequence of \cref{sub} and \cref{euler0}.

\begin{lem}\label{euler}
If $(\ve,K)$ and $(\ve',K')$ are $\vf$-regular with $\ve'\le\ve$ and $K\subseteq K'$, then
\begin{equation*}
  X_\vf^\pm(\ve,K)\subseteq X_\vf^\pm(\ve',K')
  \quad\text{and}\quad
  \chi(X_\vf^\pm(\ve',K'))\le\chi(X_\vf^\pm(\ve,K)).
\end{equation*}
Moreover, if $\chi(X_\vf(\ve,K)) = \chi(X_\vf(\ve',K'))$,
then $X_\vf^\pm(\ve',K')$ arises from $X_\vf^\pm(\ve,K)$
by attaching annuli, cross caps, and lunes
along boundary curves of $X_\vf^\pm(\ve,K)$. \qed
\end{lem}

As a direct application of \eqref{eular} and \cref{euler}, we get the

\begin{cor}\label{eulim}
There exists a $\vf$-regular pair $(\ve_\vf,K_\vf)$ such that
\begin{equation*}
  \chi(X_\vf^\pm(\ve,K))=\chi(X_\vf^\pm(\ve_\vf,K_\vf))
\end{equation*}
for all $\vf$-regular pairs $(\ve,K)$ with $\ve\le\ve_\vf$ and $K_\vf\subseteq K$. \qed
\end{cor}

Now we assume throughout that we are in the \emph{stable range},
that is, we consider $\vf$-regular pairs $(\ve,K)$ with $\ve\le\ve_\vf$
and $K_\vf\subseteq K$.
For such a pair $(\ve,K)$,
we study the isotopy type of the triples $(S,X_\vf^+(\ve,K),X_\vf^-(\ve,K))$.
Here and below we mean \emph{compactly supported topological isotopy}
when speaking of isotopy.
By the definition of $\vf$-regularity and the discussion leading to it,
the isotopy type of $(S,X_\vf^+(\ve,K),X_\vf^-(\ve,K))$ does not depend on $K$.
Hence to compare it with the isotopy type of another such triple
$(S,X_\vf^+(\ve',K'),X_\vf^-(\ve',K'))$,
we may assume that $\ve'<\ve$ and that $K$ is contained in the interior of $K'$.

Now $X_\vf^\pm(\ve,K)$ has two kinds of boundary circles:
The first kind consists of boundary circles in the interior of $K$,
the second kind consists of segments of boundary circles of $\partial K$
concatenated with segments of $\partial X_\vf^\pm(\ve)\subseteq\{\vf=\pm\ve\}$
that run inside $K$ from $\partial K$ to $\partial K$.
By the definition of $\vf$-regularity,
boundary circles of $K$ do not occur as boundary circles of  $X_\vf^\pm(\ve,K)$.
The first kind of boundary circles of  $X_\vf^\pm(\ve,K)$ is smooth,
the second kind is piecewise smooth with vertices in the points,
where the circle enters or respectively leaves $\partial K$.
The boundary of $X_\vf^\pm(\ve',K')$ consists of the corresponding two kinds
of boundary circles.

We start with a closer look at the gluings required
to obtain $X_\vf^\pm(\ve',K')$ from $X_\vf^\pm(\ve,K)$.
Since their Euler characteristics coincide,
only annuli, cross caps, and lunes are concerned; compare with \cref{sub}.
Now $\ve'<\ve$ and $K$ is contained in the interior of $K'$.
Hence the boundaries of $X_\vf^\pm(\ve,K)$ and $X_\vf^\pm(\ve',K')$
are disjoint, and therefore no lunes occur.

Suppose that a boundary circle $c$ of $X_\vf^\pm(\ve,K)$
bounds a cross cap $C$ in the complement (of the interior) of $X_\vf^\pm(\ve,K)$ in $S$.
Now $C$ decomposes $S$ into two connected regions,
the points inside $C$ and the points outside $C$.
Since $\partial C$ is contained in $\partial X_\vf^\pm(\ve,K)$,
we conclude that a curve from a point inside $C$ to a point outside
of $C\cup X_\vf^\pm(\ve,K)$ has to pass through $X_\vf^\pm(\ve,K)$.
In particular, $C$ cannot contain points on or beyond boundary circles of $K$
since otherwise it would also contain the corresponding funnels,
a contradiction to the compactness of $C$.
We conclude that in the gluing required to obtain $X_\vf^\pm(\ve',K')$ from $X_\vf^\pm(\ve,K)$,
the cross caps, including their boundary circles,
are contained in the interior of $K$.
In particular, $\vf=\pm\ve$ along their boundary circles.

Since annuli have two boundary circles,
the discussion of them involves case distinctions.
Suppose first that two boundary circles $c_0$ and $c_1$ of $X_\vf^\pm(\ve,K)$
bound a closed annulus $A$
in the complement (of the interior) of $X_\vf^\pm(\ve,K)$ in $S$.
Then any curve from a point inside $A$ to a point outside
of $A\cup X_\vf^\pm(\ve,K)$ has to pass through $X_\vf^\pm(\ve,K)$.
As in the case of cross caps,
we get that $A$ cannot contain points on or beyond boundary circles of $K$
since otherwise it would also contain the corresponding funnels,
a contradiction to the compactness of $A$.
We conclude that
in the gluing required to obtain $X_\vf^\pm(\ve',K')$ from $X_\vf^\pm(\ve,K)$,
the annuli $A$ with $\partial A\subseteq\partial X_\vf^\pm(\ve,K)$ 
are contained in the interior of $K$.
In particular, $\vf=\pm\ve$ along their boundary circles.

Finally, an annulus might be glued to $X_\vf^\pm(\ve,K)$ along one boundary circle
such that the second boundary circle belongs to the boundary of $X_\vf^\pm(\ve',K')$.
Such gluings do not change the isotopy type of $X_\vf^\pm(\ve',K')$ in $S$,
but gluings of cross caps and annuli as above do.
To remedy this, attach all annuli and cross caps to $X_\vf^\pm(\ve,K)$
which are contained in the interior of $K$ and have their boundary in $X_\vf^+(\ve,K)$
and call the resulting subsurface $S_\vf^\pm(\ve,K)$.

Note that no component of $X_\vf(\ve,K)$ is contained in any of the attached
cross caps and annuli since the components of $X_\vf(\ve,K)$ are incompressible in $S$
and their fundamental groups contain an $F_2$.
Hence
\begin{equation*}
  S_\vf^+(\ve,K) \cap S_\vf^-(\ve,K) = \emptyset.
\end{equation*}
Note also that attaching annuli and cross caps does not change the Euler characteristic.

\begin{lem}\label{charac}
If $(\ve,K),(\ve',K')$ are $\vf$-regular with $\ve'\le\ve\le\ve_\vf$
and $K_\vf\subseteq K\subseteq K'$, then
\begin{equation*}
  (S,S_\vf^+(\ve,K),S_\vf^-(\ve,K))
  \quad\text{and}\quad
  (S,S_\vf^+(\ve',K'),S_\vf^-(\ve',K'))
\end{equation*}
are isotopic in $S$.
\end{lem}

\begin{proof}
After the above discussion leading to the definition of $S_\vf^\pm(\ve,K)$,
we have the following remaining issues: 

If a boundary circle $c$ of $X_\vf^\pm(\ve,K)$ bounds a cross cap $C$ in $S$,
then either already $C\subseteq X_\vf^\pm(\ve',K')$
or else an annulus $A\subseteq C$ is attached to $c$ along one of its boundary circles
and the other boundary circle $c'$ belongs to the boundary of $X_\vf^\pm(\ve',K')$.
Then $c'$ bound a cross cap $C'$ in the complement (of the interior) of $A$ in $C$
and $C=A\cup C'$.

Conversely, if a boundary circle $c$ of $X_\vf^\pm(\ve',K')$ bounds a cross cap $C'$ in $S$,
then $c$ is contained in the interior of $K'$ and thus $\vf=\pm\ve'$ along $c$.
We conclude that $c$ is a boundary circle of an annulus $A$ attached to $X_\vf^\pm(\ve,K)$
along the other boundary circle of $A$.
Thus $C=A\cup C'$ is a cross cap in $S$ with $\partial C$
a boundary circle of $X_\vf^\pm(\ve,K)$.
By the discussion further up we obtain that $C$ is in the interior of $K$.

If boundary circles $c_0$ and $c_1$ of $X_\vf^\pm(\ve,K)$ bound an annulus $A$ in $S$,
then either already  $A\subseteq X_\vf^\pm(\ve',K')$
or else disjoint annuli $A_0,A_1\subseteq A$ are attached to $c_0$ and $c_1$,
each along one of its boundary circles, and the other boundary circles $c_0'$ and $c_1'$
bound an annulus $A'\subseteq A$ between $A_0$ and $A_1$.
Then $A=A_0\cup A'\cup A_1$.

Conversely, if boundary circles $c_0$ and $c_1$ of $X_\vf^\pm(\ve',K')$
bound an annulus $A'$ in $S$,
then $A'$ is contained in the interior of $K'$ and thus $\vf=\pm\ve'$ along $\partial A'$.
Arguing as in the case of cross caps, we get annuli $A_0$ and $A_1$
with one boundary circle in $X_\vf^\pm(\ve,K)$ and the other equal to $c_0$ and $c_1$,
respectively.
Thus $A=A_0\cup A'\cup A_1$ is an annulus in $S$ such that $\partial A$
lies in $X_\vf^\pm(\ve,K)$. 
By the discussion further up we obtain that $A$ is in the interior of $K$.
\end{proof}

We call the isotopy type of the triple $(S,S_\vf^+(\ve,K),S_\vf^-(\ve,K))$
the \emph{type of $\vf$} and the Euler characteristic of $S_\vf(\ve,K)$
the \emph{characteristic of $\vf$}.

\begin{lem}\label{isoto}
If $\psi$ is a non-trivial finite linear combination of eigenfunctions of $S$
with corresponding eigenvalues $\le\Lambda(S)$,
with the same characteristic as $\vf$, and suffciently close to $\vf$,
then the types of $\vf$ and $\psi$ coincide.
\end{lem}

\begin{proof}
Let $L$ be a compact neighborhood of $K$
which contains all the $\ve$-discs with respect to $\vf$ which intersect $K$.
Consider a function $\psi$ with the same characteristic as $\vf$
which is $C^2$-close to $\vf$ on $L$.
Then $\pm\ve$ are regular values of $\psi|_L$,
the curves $\psi|_L=\pm\ve$ intersect $\partial K$ transversally,
and there is a small isotopy of $S$ which leaves $K$ and $\partial K$ invariant
which deforms the configuration of curves $\{\psi=\ve\}\cap L$ and $\{\psi=-\ve\}\cap L$
to the configuration of curves $\{\vf=\ve\}\cap L$ respectively $\{\vf=-\ve\}\cap L$
and, therefore, also the subsurfaces $\{\psi\ge\ve\}\cap K$ and $\{\psi\le-\ve\}\cap K$
to the subsurfaces $\{\vf\ge\ve\}\cap K$ respectively $\{\vf\le-\ve\}\cap K$.

Clearly, if a boundary segment of the latter intersects an $\ve$-disc $D$ of $\vf$,
then $D$ is contained in $L$ and corresponds under the isotopy to an $\ve$-disc $B$ of $\psi$.
Attaching the parts $B\cap K$ of such discs, we get a surface $T^\pm$
such that the above isotopy deforms $T^\pm$ to $X_\vf^\pm(\ve,K)$.
In particular, the fundamental group of $T^\pm$ contains an $F_2$,
$T^\pm$ is incompressible in $S$ and
\begin{equation*}
  \chi(T^\pm)=\chi(X_\vf^\pm(\ve,K)).
\end{equation*}
By changing $\ve$ slightly, we can achieve that $\ve$ is also $\psi$-regular.
Then, by what we said, $T^\pm$ is a component of $X_\psi^\pm(\ve,K)$.
Moreover, choosing a $\psi$-regular $(\ve,K')$ with $K$ in the interior of $K'$,
we have $T^\pm\subseteq X_\psi^\pm(\ve,K')$.
Hence $X_\psi^\pm(\ve,K')$ is obtained from $T^\pm$ by attaching annuli,
cross caps, and lunes.
Now annuli where both boundary curves are attached to $T^\pm$ and cross caps
attached to $T^\pm$ are contained in the interior of $K$ and belong to $S_\psi^\pm(\ve,K)$.
We (finally) conclude that $(S,S_\psi^+(\ve,K'),S_\psi^-(\ve,K'))$ is isotopic to
the triple $(S,S_\vf^+(\ve,K),S_\vf^-(\ve,K))$.
\end{proof}


\begin{proof}[End of proof of \cref{small}]
Let $\mathbb E$ be a subspace of $L^2(M)$
which is generated by finitely many eigenfunctions
with corresponding eigenvalues $\le\Lambda(S)$
and denote by $\mathbb S$ the unit sphere in $\mathbb E$
and by $\mathbb P$ the projective space of $\mathbb E$.
\cref{small} follows if any such $\mathbb E$ has dimension at most $-\chi(S)$.

Since $\chi(S_\vf^\pm(\ve,K))=\chi(X_\vf^\pm(\ve,K))$,
\eqref{eular} and \cref{charac} imply that we obtain a partition of $\mathbb S$
into the subsets $\mathbb A_i$ consisting of functions $\vf$
with characteristic $i\in\{-\chi(S),\dots,-1\}$.
By definition, $\vf\in\mathbb A_i$ if and only if $-\vf\in\mathbb A_i$.
Hence the partition of $\mathbb S$ into the sets $\mathbb A_i$
is the preimage of a partition of $\mathbb P$ into subsets $\mathbb B_i$
under the covering projection $\pi\colon\mathbb B\to\mathbb P$.

Now at least one of the subsurfaces $S_\vf^+(\ve,K)$ or $S_\vf^-(\ve,K)$ is nonempty
and contains two loops $c_0$ and $c_1$ with intersection number one.
Then the image of $c_1$ under an isotopy of $S$ will still intersect $c_0$,
and therefore there is no isotopy of $S$ which interchanges the disjoint
subsurfaces $S_\vf^+(\ve,K)$ and $S_\vf^-(\ve,K)$.
Hence the type of $\vf\in\mathbb S$ is different from the type of $-\vf$.
Hence by \cref{isoto}, the covering $\pi$ is trivial over the subsets $\mathbb A_i$.
Now $\mathbb P$ cannot be covered by less than $\dim\mathbb E$ subsets
over which $\pi$ is trivial, by Lemma 8 in \cite{Se}.
We conclude that $\dim\mathbb E\le-\chi(S)$.
\end{proof}



\end{document}